\documentclass[a4paper,fleqn]{cas-sc}
\usepackage[numbers]{natbib}
\usepackage{lineno}
\usepackage{hyperref}
\usepackage{enumerate}
\usepackage{graphicx}
\usepackage{mathtools}
\usepackage{float}
\usepackage{placeins}
\usepackage{caption}
\usepackage{amsthm, amsmath, amssymb, amsfonts}
\newtheorem{theorem}{Theorem}[section]
\newdefinition{definition}{Definition}[section]     
\newtheorem{lemma}{Lemma}[section]
\newtheorem{proposition}{Proposition}[section]

\newenvironment{enumlemma}[2][(i)]
{\begin{lemma}#2\mbox{}
		\begin{enumerate}[#1]}
		{\end{enumerate}
\end{lemma}}

\newenvironment{enumproof}[2][(i)]
{\begin{proof}#2%\mbox{}
		\begin{enumerate}[#1]}
		{\end{enumerate}
\end{proof}}                                                       
\def\tsc#1{\csdef{#1}{\textsc{\lowercase{#1}}\xspace}}
\tsc{WGM}
\tsc{QE}
\begin{document}
	\title[mode = title]{Bivariate fractal interpolation functions on triangular domain for numerical integration and approximation} 	
	\author[1]{Aparna M.P} \fnmark[1]	
	\author[2]{P. Paramanathan}[orcid=0000-0003-0688-4858]
	\cormark[1]
	\fnmark[2]
	\address[1]{Department of Mathematics, 
			Amrita School of Physical Sciences, Coimbatore,
		Amrita Vishwa Vidyapeetham, India.}
	\address[2]{Department of Mathematics, 
			Amrita School of Physical Sciences, Coimbatore, Amrita Vishwa Vidyapeetham, India.}
	
	\cortext[cor1]{Corresponding author}
	\fntext[fn1]{mp$\_$aparna@cb.students.amrita.edu (Aparna M.P)} \fntext[fn2]{p$\_$paramanathan@cb.amrita.edu (P. Paramanathan)}
	\begin{abstract}
		The primary objectives of this paper are to present the construction of bivariate fractal interpolation functions over triangular interpolating domain using the concept of vertex coloring and to propose a double integration formula for the constructed interpolation functions. 
		Unlike the conventional constructions, each vertex in the partition of the triangular region has been assigned a color such that the chromatic number of the partition is 3. 
		A new method for the partitioning of the triangle is proposed with a result concerning the chromatic number of its graph.   Following the construction, a formula determining the vertical scaling factor is provided. With the newly defined vertical scaling factor, it is clearly observed that the value of the double integral coincides with the integral value calculated using fractal theory. Further, a relation connecting the fractal interpolation function with the equation of the plane passing through the vertices of the triangle is established. Convergence of the proposed method to the actual integral value is proven with sufficient lemmas and theorems. Sufficient examples are also provided to illustrate the method of construction and to verify the formula of double integration.
	\end{abstract}
	
	\begin{keywords}
		Bivariate fractal interpolation function (BFIF) \sep Chromatic number  \sep Double integration 
	\end{keywords}
	
	\maketitle
	\section{Introduction}
	Fractal interpolation functions (FIF) are constructed using iterated function system (IFS)\cite{fe}. The IFS consists of a complete metric space together with a finite set of contraction mappings \cite{an}. Normally, each contraction map in the IFS is composed of two types of functions, the former one responsible for the contraction of the entire interpolating domain and the latter one for defining the Read-Bejraktarevic operator whose fixed point is the required fractal interpolation function. The fundamental problem in the theory of fractal interpolation is to  establish the well definiteness of this operator. In \cite{bf}, L. Dalla imposed a restrictive condition on the interpolation points for tackling this problem, when the interpolating domain is a rectangle. For the same purpose, a piecewise function, defined in terms of the usual IFS, is proposed in \cite{nt}. For the triangular interpolating domain, the problem of well definiteness was dealt by Geronimo and Hardin and they proposed the method of vertex coloring to solve the problem \cite{fi}. The present paper further connects this approach to define  bivariate fractal interpolation functions and to derive double integration for such functions. 
	\\
	
	Barnsley in \cite{ff} introduced the idea of fractal interpolation functions using iterated function systems for the first time where single variable interpolation functions were generated as the attractors of the IFS. The construction of bivariate fractal interpolation functions considered by L Dalla required the interpolation points on the boundary of the rectangle to be collinear \cite{bf}.  Malysz used the fold-out technique for constructing  bivariate fractal interpolation function over rectangles by taking the same vertical scaling factor \cite{md}. In \cite{cf}, considering the vertical scaling factor as a function, Metzler and Yun generalised this construction. In \cite{fs}, Massopust constructed bivariate fractal interpolation functions on a triangular region with a restrictive condition on the interpolation points. The construction proposed in his work was further modified in \cite{fi}. \cite{fi} introduces the idea of the coloring of the vertices to solve the problem of well definiteness. 
	The construction of the bivariate fractal interpolation function is again considered in \cite{fii}. The paper, however, fails to establish the well definiteness of the fractal interpolation operator. The numerical integration of fractal interpolation functions was first carried out by Navascues in \cite{ni}. The derived formula for integration is then compared with the compound trapezoidal rule there.		 \\
	
	The present paper aims to define double integration for two variable fractal interpolation functions constructed over a triangular domain. By proposing a method for the partition of the triangle and introducing a new vertical scaling factor, this paper provides a detailed explanation for the construction of these functions. Following the derivation of double integration, the constructed fractal interpolation function is approximated to the equation of the plane passing through the vertices of the triangle. After proving the theorems in error analysis using this approximation, the paper shows the attractors of the IFS's and the double integral values obtained for some functions.
	\\
	
	The organization of the paper is as follows: The second section presents the formulation of the IFS and proves the corresponding theorems with the results concerning the partitioning of the triangle and its chromatic number. The derivation of the double integration formula is proposed in the third section. In the fourth section, it is established that the bivariate fractal interpolation functions defined over triangular regions can be approximated to the equation of the plane passing through the vertices of the triangle. The fifth section provides the formula for the vertical scaling factor and its upper bound. Using, the newly obtained approximating function, the propositions and theorems are proved for the error analysis. The paper concludes by displaying the tables and graphs describing the results.
	
	\section{Construction of bivariate fractal interpolation function over triangular regions}
	\vspace{0.2cm}
	\subsection{Method of Partition}
	\vspace{0.2cm}
	
	Consider the triangle $D$ with vertices $A(x_{1},y_{1}), B(x_{2},y_{2}),C(x_{3},y_{3}).$ The algorithm for the partition is as follows.
	\begin{enumerate}
		\item Divide the height of the triangle $D$ into $'d'$ number of equal parts, thereby getting $'d+1'$ new points $(x_{i},y_{1}), (x_{i},y_{2}), ..., (x_{i},y_{d+1})$ along the height of $D,$ where $y_{1}=y$ coordinate of the point $A$ or $B,$ $y_{d+1}=y$ coordinate of the point $C.$
		\item Draw lines $y=y_{j}, j=1,2,...,d$ parallel to $X-$axis from $AC$ to $BC.$
		\item Divide each of the horizontal lines $y=y_{j}, j=1,2,...,d$ into $'d'$ number of equal parts, generating $'d+1'$ new points denoted by $(x_{1},y_{j}), (x_{2},y_{j}),...(x_{d+1},y_{j}) $ along each horizontal line $y=y_{j}$ for $j=1,2,...,d$ where $(x_{1},y_{j}), (x_{d+1},y_{j})$ lies on the sides $AC$ and $BC$ respectively. 
		\item Then, join the new points as shown in Figure \ref{Figure 1}. 
	\end{enumerate}
	
	\begin{center}
		\begin{figure}[h]
			\begin{center}
				\includegraphics[width=12cm]{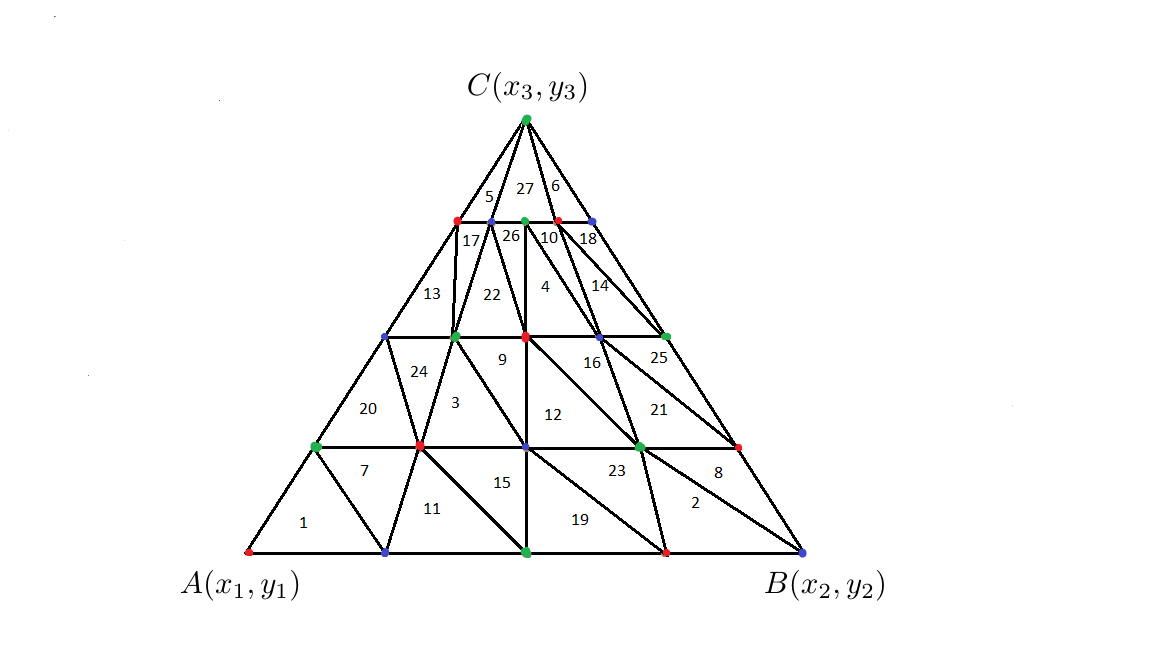}
				\caption{Partition of $D$ when $d=4$}
				\label{Figure 1}
			\end{center}
		\end{figure}
	\end{center}
	In Figure \ref{Figure 1}, red corresponds to color 1, blue stands for color 2, and green for color 3.
	\begin{enumlemma}[i)]
		{For the partition defined above, if the height of the triangle $D$ is divided into $d$ number of equal parts, then, there are}
		\item  $2d^{2}-2d+3$ subtriangles and
		\item  $d^{2}+d+1$ vertices, 
	\end{enumlemma}
	in the partition.

	\begin{enumproof}[i)]
		
		{}\item According to the partition defined, the height of the triangle is divided into $d$ number of equal parts, resulting $d+1$ $'y'$ values, where $y_{d+1}$ is the $y$ coordinate of the top vertex of $D.$  Along each horizontal line $y=y_{j}, j=1,2,...,d$ the line is divided into $d$ number of equal parts. The coordinates of the newly obtained points along the line $y=y_{j}$ are denoted by $(x_{1},y_{j}), (x_{2},y_{j}), ..., (x_{d+1},y_{j}).$ It is observed from the partition that between two consecutive points on the line $y=y_{j},$ two subtriangles are obtained (a normal and an inverted triangle). Hence, along each line $y=y_{j},$ there are $2d$ subtriangles for $j=2,...,d.$ Considering the top most three subtriangles, which are fixed irrespective of $d,$ there are $(2d)(d-1)+3$ subtriangles in the partition. $i.e,$ the total number of subtriangles in the partition is $2d^{2}-2d+3.$ Hence the proof.
		\item Following the partition, since each horizontal line $y=y_{j}$ is divided into $d$ equal parts, there are $d+1$ new vertices along each line $y=y_{j}, j=1,2,...,d.$ Now, including the top most vertex $C,$ of the triangle $D,$ there are $[(d+1)d]+1=d^{2}+d+1$ vertices in the partition. Hence the proof.
	\end{enumproof}

	\begin{lemma}
		If $d=3n+1, n \in N, $ then the graph of the partition defined above has chromatic number 3.
	\end{lemma}
	\begin{proof}
		According to the partition, each horizontal line $y=y_{j}$ is divided into $d$ equal parts, generating $d+1$ points along that line. Let the points be denoted by $(x_{1},y_{j}), (x_{2},y_{j}), ..., (x_{d+1},y_{j}).$ Among these points,  $(x_{1},y_{j})$ and $(x_{d+1},y_{j})$ are points on the two sides of the triangle $D.$ The remaining points are intermediate points and there $d-1$ such points along that line, where $d-1$ is a multiple of 3. Now, considering the coloring of theses points with the least number of colors, since the points  $(x_{1},y_{1})$ and $(x_{d+1},y_{1})$  are adjacent with respect to the triangle $D,$ they have to be colored differently. Without loss of generality, let $(x_{1},y_{1})$ and $(x_{d+1},y_{1})$ be colored with colors '1' and '2' respectively. Now, since the point $(x_{2},y_{1})$ is adjacent to $(x_{1},y_{1}),$ it should be colored '2.' Similarly, $(x_{3},y_{1})$ is colored '1'. Proceeding in this manner, the point $(x_{d-1},y_{1})$ will be colored '1.' Then, the point $(x_{d},y_{1})$ has to colored with a different color other than '1' and '2', since it is adjacent to both $(x_{d-1},y_{1})$ and $(x_{d+1},y_{1}).$ Hence, a minimum of 3 colors are needed to color the points along this line. The same reasoning can be applied along each of the horizontal lines $y=y_{j}$ and along the slanting sides of $D.$ Thus, it can be established that atleast 3 different colors are needed to properly color the graph. Therefore, the chromatic number of the graph is 3.
	\end{proof}
	\subsection{Construction}
	\vspace{0.2cm}
	Consider a triangular domain $D$ with vertices $(x_{1},y_{1}), \,\, (x_{2},y_{2}), \,\, (x_{3},y_{3}),$ colored '1', '2' and '3' respectively. 
	Let the triangle be partitioned into $N$ number of subtriangles $D_{1}, D_{2},..., D_{N}$ such that $D=\cup_{n=1}^{N}D_{n} $ and $\cap_{n=1}^{N}D_{n}=\phi.$ The partitioning is done such that the chromatic number of their corresponding graph is 3. Each subtriangle be numbered from 1 to $N$ as shown in Figure \ref{Figure 1}. Set $P=\{(x_{nj},y_{nj}): j=1,2,3, n=1,2,...,N\}$ to be the set of all vertices of the subtriangles $D_{n}, n=1,2,...,N.$ Let $z_{nj}=f(x_{nj},y_{nj})$ be the corresponding function values. \\ Let $R=\{(x_{nj},y_{nj},z_{nj}): j=1,2,3, n=1,2,...,N \}$ be the data set. Without loss of generality, let $(x_{n1},y_{n1})$ denotes the vertex colored '1',  $(x_{n2},y_{n2})$  be the vertex with color '2', and $(x_{n3},y_{n3})$ be the vertex colored '3' in $D_{n}.$
	\\
	Consider an invertible, affine map $L_{n}:D\rightarrow D_{n}$ such that 
	\begin{align}\label{uy}
		L_{n}(x_{j},y_{j})&=(x_{nj},y_{nj}), \,\,\text{for} \,\,\, j=1,2,3.
	\end{align}
	
	\noindent	$i.e,$ $L_{n}$ maps $(x_{j},y_{j})$ to the vertex in $D_{n}$, which is colored $j,$ for $j=1,2,3.$
	The map $L_{n}$ used here is given by,
	\begin{align}
		L_{n}(x,y)=\begin{bmatrix}
			\alpha_{n1} & \alpha_{n2} \\
			\alpha_{n3} & \alpha_{n4}
		\end{bmatrix} 
		\begin{bmatrix}
			x \\ y
		\end{bmatrix} + 
		\begin{bmatrix}
			\beta_{n1} \\ \beta_{n2}
		\end{bmatrix} 
	\end{align}
	
	\noindent	Choose a vertical scaling factor $\alpha_{n7}$ between -1 and 1 for $n=1,2,...,N$ with the scale vector $\overline{\alpha_{n7}}.$ \\ Set $F=D\times R$  and consider $N$ maps $F_{n},$ contractive in the third variable such that 
	\begin{align}\label{ur}
		F_{n}(x_{j},y_{j},z_{j})&=z_{nj}, \,\, \text{for} \,\,\, j=1,2,3, \,\,\, n=1,2,...,N.
	\end{align}
	The map $F_{n}$ is given by,
	\begin{align}
		F_{n}(x,y,z)=Q_{n}(x,y)+\alpha_{n7}z, \,\,\, n=1,2,...,N, 
	\end{align}
	\noindent where $ \,\,\, Q_{n}(x,y)=\alpha_{n5}x+\alpha_{n6}y+\beta_{n3}, \,\,\, n=1,2,...,N.$
	\vspace{0.1cm}

	\noindent 	Then, the IFS becomes 	
	\begin{align*}
		w_{n}(x,y,z)=\big(L_{n}(x,y),F_{n}(x,y,z)\big), n=1,2,...,N.
	\end{align*}
	In matrix notation,
	\begin{align}
		W_{n}(x,y,z)= \begin{bmatrix}
			\alpha_{n1} & \alpha_{n2} & 0 \\
			\alpha_{n3} & \alpha_{n4} & 0 \\
			\alpha_{n5} &\alpha_{n6} &\alpha_{n7} 		
		\end{bmatrix}
		\begin{bmatrix}
			x \\y\\z
		\end{bmatrix} +
		\begin{bmatrix}
			\beta_{n1}\\ \beta_{n2}\\\beta_{n3}
		\end{bmatrix}
	\end{align}
	
	\noindent The constants in the matrix are obtained by solving the endpoint conditions \eqref{uy} and \eqref{ur}. Then, 
	\begin{align*}
		\alpha_{n1}&=\frac{(x_{n1}-x_{n2})(y_{1}-y_{3})-(x_{n1}-x_{n3})(y_{1}-y_{2})}{(x_{1}-x_{2})(y_{1}-y_{3})-(x_{1}-x_{3})(y_{1}-y_{2})} \\
		\alpha_{n2}&=\frac{(x_{n1}-x_{n3})(x_{1}-x_{2})-(x_{n1}-x_{n2})(x_{1}-x_{3})}{(x_{1}-x_{2})(y_{1}-y_{3})-(x_{1}-x_{3})(y_{1}-y_{2})}\\
		\alpha_{n3}&=\frac{(y_{n1}-y_{n2})(y_{1}-y_{3})-(y_{n1}-y_{n3})(y_{1}-y_{2})}{(x_{1}-x_{2})(y_{1}-y_{3})-(x_{1}-x_{3})(y_{1}-y_{2})}\\
		\alpha_{n4}&=\frac{(y_{n1}-y_{n3})(x_{1}-x_{2})-(y_{n1}-y_{n2})(x_{1}-x_{3})}{(x_{1}-x_{2})(y_{1}-y_{3})-(x_{1}-x_{3})(y_{1}-y_{2})}\\
		\alpha_{n5}&=\frac{(z_{n1}-\alpha_{n7}z_{1})(y_{2}-y_{3})+(z_{n2}-\alpha_{n7}z_{2})(y_{3}-y_{1})+(z_{n3}-\alpha_{n7}z_{3})(y_{1}-y_{2})}{(x_{1}-x_{2})(y_{1}-y_{3})-(x_{1}-x_{3})(y_{1}-y_{2})}\\
		\alpha_{n6}&=\frac{-( (z_{n1}-\alpha_{n7}z_{1})(x_{2}-x_{3})+ (z_{n2}-\alpha_{n7}z_{2})(x_{3}-x_{1})+ (z_{n3}-\alpha_{n7}z_{3})(x_{1}-x_{2}))}{(x_{1}-x_{2})(y_{1}-y_{3})-(x_{1}-x_{3})(y_{1}-y_{2})}\\
		\beta_{n1}&=\frac{(x_{n1}x_{2}-x_{1}x_{n2})y_{3}+ (x_{n3}x_{1}-x_{3}x_{n1})y_{2}+ (x_{n2}x_{3}-x_{2}x_{n3})y_{1}}{(x_{1}-x_{2})(y_{1}-y_{3})-(x_{1}-x_{3})(y_{1}-y_{2})}\\
		\beta_{n2}&=\frac{(x_{2}y_{3}-x_{3}y_{2})y_{n1}+ (x_{3}y_{1}-x_{1}y_{3})y_{n2}+(x_{1}y_{2}-x_{2}y_{1})y_{n3}}{(x_{1}-x_{2})(y_{1}-y_{3})-(x_{1}-x_{3})(y_{1}-y_{2})}\\
		\beta_{n3}&=\frac{(z_{n1}-\alpha_{n7}z_{1})(x_{2}y_{3}-x_{3}y_{2}) + (z_{n2}-\alpha_{n7}z_{2})(x_{3}y_{1}-x_{1}y_{3})+(z_{n3}-\alpha_{n7}z_{3}) (x_{1}y_{2}-x_{2}y_{1})}{(x_{1}-x_{2})(y_{1}-y_{3})-(x_{1}-x_{3})(y_{1}-y_{2})}
	\end{align*}

	\begin{proposition}
		There is a metric $\sigma$ equivalent to the Euclidean metric on $R^{3}$ such that the above defined IFS is hyperbolic with respect to the new metric. Moreover, there is a unique, non-empty, compact set $G$ such that $$G=\cup_{n=1}^{N}w_{n}(G),$$ called the attractor of the IFS.	
	\end{proposition}
	
	\begin{proof}
		Let $(x,y,z), \,\, (x^{'},y^{'},z^{'})$ be two points in $D.$ Consider the metric $\sigma$ given by 
		\begin{align*}
			\sigma(x,y,z) &=|x-x^{'}|+|y-y^{'}|+\theta|z-z^{'}|,
		\end{align*} where $\theta$ is to be specified later. Clearly, $\sigma$ is equivalent to Euclidean metric on $R^{3}.$ \\
		Now, \vspace{0.1cm}
		\\ 
		$\sigma\big(w_{n}(x,y,z), w_{n}(x^{'},y^{'},z^{'})\big)$
		\vspace{-0.2cm}
		\begin{flalign}
			\nonumber
			&= \sigma\Big(\big(L_{n}(x,y), F_{n}(x,y,z)\big),\big(L_{n}(x^{'},y^{'}), F_{n}(x^{'},y^{'},z^{'})\big)\Big) \\
			\nonumber		
			&=\sigma\Big( \big((\alpha_{n1}x+\alpha_{n2}y+\beta_{n1}, \alpha_{n3}x+\alpha_{n4}y+\beta_{n2}), (\alpha_{n5}x+\alpha_{n6}y+\alpha_{n7}z+\beta_{n3})\big), \\ \nonumber &  \big((\alpha_{n1}x^{'}+\alpha_{n2}y^{'}+\beta_{n1}, \alpha_{n3}x^{'}+\alpha_{n4}y^{'}+\beta_{n2}), (\alpha_{n5}x^{'}+\alpha_{n6}y^{'}+\alpha_{n7}z^{'}+\beta_{n3})\big) \Big)\\
			\nonumber
			&\leq\Big(|\alpha_{n1}|+|\alpha_{n3}|\Big)|x-x^{'}|+\Big(|\alpha_{n2}|+|\alpha_{n4}|\Big)|y-y^{'}|+\theta\Big(|\alpha_{n5}||x-x^{'}|+|\alpha_{n6}||y-y^{'}|+|\alpha_{n7}||z-z^{'}|\Big)\\
			\nonumber 
			&=\Big(|\alpha_{n1}|+|\alpha_{n3}|+\theta|\alpha_{n5}|\Big)|x-x^{'}|+\Big(|\alpha_{n2}|+|\alpha_{n4}|+\theta|\alpha_{n6}|\Big)|y-y^{'}|+\theta|\alpha_{n7}||z-z^{'}|
		\end{flalign}
		
		Choose $\theta=\text{min}\left\{ \frac{\text{min}\{ 0.5-(|\alpha_{n3}|+|\alpha_{n1}|)\}}{\text{max}\{|\alpha_{n5}|\}}, \frac{\text{min}\{ 0.5-(|\alpha_{n2}|+|\alpha_{n4}|)\}}{\text{max}\{|\alpha_{n6}|\}} \right\}.$
		Then, the above expression becomes, \vspace{0.2cm}\\ 
		$\sigma\big(w_{n}(x,y,z), w_{n}(x^{'},y^{'},z^{'})\big)$
		\vspace{-0.2cm}
		\begin{align}
			\nonumber
			& \leq \Big(|\alpha_{n1}|+|\alpha_{n3}|+\theta|\alpha_{n5}|\Big)|x-x^{'}|+\Big(|\alpha_{n2}|+|\alpha_{n4}|+\theta|\alpha_{n6}|\Big)|y-y^{'}|+\theta|\alpha_{n7}||z-z^{'}|\\		
			\nonumber
			&\leq 0.5|x-x^{'}|+0.5|y-y^{'}|+\theta|z-z^{'}||\alpha_{n7}| \\
			\nonumber
			&\leq 0.5|x-x^{'}|+0.5|y-y^{'}|+\theta|z-z^{'}|\text{max}_{n}\{|\alpha_{n7}|\} \\
			\nonumber
			&\leq \text{max}\big\{0.5,\text{max}_{n}\{|\alpha_{n7}|\}\big\} \sigma\big((x,y,z),(x^{'},y^{'},z^{'})\big) \\
			\nonumber 
			& <1.
		\end{align}
		
		\noindent 	since $|\alpha_{n7}|<1, \text{max}\{0.5,|\text{max}_{n}\{|\alpha_{n7}|\}\} <1.$
		Hence the IFS is hyperbolic. Then, trivially, there exists a unique, nonempty, compact set $G$ in $R^{3}$ such that $$G=\cup_{n=1}^{N}w_{n}(G).$$
	\end{proof}

	\begin{theorem}
		Let the given data set be $R=\{(x_{nj},y_{nj},z_{nj}): j=1,2,3, n=1,2,...,N \}.$ Consider the IFS defined above, associated with the data set, $w_{n}(x,y,z)=\big(L_{n}(x,y),F_{n}(x,y,z)\big)$ for $n=1,2,...,N.$ Choose $-1<\alpha_{n7}<1$ so as to make the IFS hyperboilc. Let $G$ denote the attractor of the IFS. Then, $G$ is the graph of a continuous function $f:D\rightarrow R$ such that $f$ interpolates the given data set. 
	\end{theorem}
	
	\begin{proof}
		Let $\mathbb{F}$ denote the set of all continuous functions $f:D\rightarrow R$ such that % $f(x_{j},y_{j})=z_{j}, j=1,2,3$ and
		$f(x_{nj}, y_{nj})=z_{nj}, $ for $n=1,2,..,N$ and $j=1,2,3.$ Then, it is easily verified that $\mathbb{F}$ is a complete metric space with the supremum metric. Consider the operator $T:\mathbb{F} \rightarrow \mathbb{F}$ defined by 
		\begin{align}
			(Tf)(x,y)=F_{n}\big(L_{n}^{-1}(x,y), foL_{n}^{-1}(x,y)\big), \,\,\, (x,y)\in  D_{n}, \,\,\, n=1,2..,N.
		\end{align}
		Trivially, $T$ is well defined at the interiors of each subtriangle. Now, it remains to establish the well definiteness of $T$ at the common edges. Let $\overline{((x_{nj},y_{nj}),(x_{nj^{'}},y_{nj^{'}}))}$ be an arbitrary edge of the subtriangle $D_{n}.$ Let it be shared by the subtriangle $D_{m}.$ Considering the vertex $(x_{nj},y_{nj})$ in $D_{n}.$ Then, $L_{n}^{-1}(x_{nj},y_{nj})=(x_{j},y_{j}).$ Similarly, considering $(x_{nj},y_{nj})$ as in $D_{m},$ it will be denoted by $(x_{mj}, y_{mj}).$ 
		Also, $L_{m}^{-1}(x_{mj},y_{mj})=(x_{j},y_{j}).$ The same can be done for the vertex $(x_{nj^{'}},y_{nj^{'}}).$ 
		Now, since $L_{n}^{-1}, L_{m}^{-1}$ are affine, they give the same output along the edge $\overline{(x_{nj},y_{nj}),(x_{nj^{'}},y_{nj^{'}})}.$ To verify the same for $T,$ consider the vertex $(x_{nj},y_{nj})$ in the edge $\overline{((x_{nj},y_{nj}),(x_{nj^{'}},y_{nj^{'}}))}.$ Let this edge be shared by $D_{n} $ and $D_{m}.$ Now, considering $(x_{nj},y_{nj})$ as a point in $D_{n}, $
		\begin{align*}
			(Tf)(x_{nj},y_{nj}) &= F_{n}\big(L_{n}^{-1}(x_{nj},y_{nj}), foL_{n}^{-1}(x_{nj},y_{nj})\big) \\
			\nonumber
			&=F_{n}\big((x_{j},y_{j}), f(x_{j},y_{j})\big) \\
			\nonumber
			&=F_{n}\big(x_{j},y_{j},z_{j}\big)\\
			\nonumber
			&=z_{nj}
		\end{align*}

		Similarly, considering  $(x_{nj},y_{nj})$ as a point in $D_{m}, $ then, it will be denoted by $(x_{mj},y_{mj}).$
		Then, \begin{align}
			\nonumber
			(Tf)(x_{nj},y_{nj}) &= (Tf)(x_{mj},y_{mj})\\
			\nonumber
			&=F_{m}\big(L_{m}^{-1}(x_{mj},y_{mj}), foL_{m}^{-1}(x_{mj},y_{mj})\big) \\
			\nonumber
			&=F_{m}\big((x_{j},y_{j}), f(x_{j},y_{j})\big) \\
			\nonumber
			&=F_{m}\big(x_{j},y_{j},z_{j}\big)\\
			\nonumber
			&=z_{mj} \\
			\nonumber
			&=f(x_{mj}, y_{mj}) \\
			\nonumber
			&=f(x_{nj}, y_{nj}) \\
			\nonumber
			&=z_{nj}
		\end{align}   
		Similarly, $(Tf)(x_{nj^{'}},y_{nj^{'}})=z_{nj^{'}},$ while taking the point $(x_{nj^{'}},y_{nj^{'}})$ as in $D_{n}$ and $D_{m}.$ 
		Eventually, since $T$ is defined in terms of $L_{n}$ and $L_{n}$ is affine, invertible map, it follows that $T$ is well defined along each edge of the subtriangles. 
		
		Thus, $T$ is a well defined map. Hence the operator $T$ is continuous. \\
		In order to show that the operator $T$ satisfies the endpoint conditions, consider an arbitrary point $(x_{nj},y_{nj}).$ Then, 
		\begin{align*}
			(Tf)(x_{nj},y_{nj}) &= F_{n}\big(L_{n}^{-1}(x_{nj},y_{nj}), foL_{n}^{-1}(x_{nj},y_{nj})\big) \\
			\nonumber
			&=F_{n}\big((x_{j},y_{j}), f(x_{j},y_{j})\big) \\
			\nonumber
			&=F_{n}\big(x_{j},y_{j},z_{j}\big)\\
			\nonumber
			&=z_{nj}
		\end{align*}
		Since $T$ is well defined, it also gives the same value while considering $(x_{nj},y_{nj})$ as a point in $D_{m}.$

		\noindent In order to prove contractivity of $T,$ let $f,g$ be two points in $\mathbb{F}.$ Then,	
		\begin{align}		
			\nonumber
			d\big(Tf(x,y),Tg(x,y)\big) 
			&= \text{sup} \big\{|Tf(x,y)-Tg(x,y)|:(x,y)\in D\big\} \\
			\nonumber
			&= \text{max}_{n} \text{sup} \big\{|Tf(x,y)-Tg(x,y)|:(x,y)\in D_{n}\big\} \\
			\nonumber
			&= \text{max}_{n} \text{sup} \big\{|F_{n}\big(L_{n}^{-1}(x,y), foL_{n}^{-1}(x,y)\big) - F_{n}\big(L_{n}^{-1}(x,y), goL_{n}^{-1}(x,y)\big)|: (x,y)\in D_{n} \big\} \\
			\nonumber
			&= \text{max}_{n} \text{sup} \big\{ |\alpha_{n7}\big(foL_{n}^{-1}(x,y)-goL_{n}^{-1}(x,y)\big)|: (x,y)\in D_{n}\big\} \\
			\nonumber
			& \leq \text{max}_{n} \{|\alpha_{n7}|\} \text{sup} \big\{|(f(x,y)-g(x,y))|: (x,y)\in D\big\} \\
			\nonumber
			& \leq \alpha d\big(f(x,y),g(x,y)\big), 
		\end{align} where $\alpha=\text{max}_{n} \{|\alpha_{n7}|\} <1.$
		
		By contraction mapping theorem, $T$ has a unique, fixed point $f$ such that $Tf(x,y)=f(x,y).$ $i.e,$ 
		\begin{align}
			\nonumber
			f(x,y) &= F_{n}\big(L_{n}^{-1}(x,y), foL_{n}^{-1}(x,y)\big) \\
			&= \alpha_{n7}foL_{n}^{-1}(x,y) +Q_{n}oL_{n}^{-1}(x,y) \label{rec}
		\end{align}
		Clearly, $f$ passes through the interpolation points. Now, let $G$ be the unique, attractor of the IFS and $G^{'}$ be the graph of $f.$ Then,

		\begin{align}
			\nonumber
			W(G^{'}) &=\cup_{n=1}^{N}\big(w_{n}(G^{'})\big) \\ \nonumber
			&=\cup_{n=1}^{N}w_{n}\big(\big\{(x,y,f(x,y)):(x,y)\in D\big\}\big)\\
			\nonumber
			&=\cup_{n=1}^{N}\big\{\big(L_{n}(x,y),F_{n}(x,y,f(x,y))\big):(x,y)\in D\big\}\\
			\nonumber
			&=\cup_{n=1}^{N}\big\{\big(L_{n}(x,y),foL_{n}(x,y)\big):(x,y)\in D\big\}\\	\nonumber
			&=G^{'}.
		\end{align}	 
		Thus, graph of $f$ is the unique, attractor of the IFS.
	\end{proof}
	
	\section{Double integration}
	Let $M$ denote the double integral of a bivariate fractal interpolation function $f$ over the triangular region $D.$  Then, 
	
	\begin{align}
		\nonumber
		M&= \int_{D} f(x,y) dxdy \\ \nonumber &= \sum_{n=1}^{N} \int_{D_{n}} f(x,y) dxdy \\
		\nonumber
		&=\sum_{n=1}^{N} \int_{D_{n}} Tf(x,y) dxdy \\
		\nonumber
		&=\sum_{n=1}^{N} \int_{D_{n}} F_{n}\big(L_{n}^{-1}(x,y), foL_{n}^{-1}(x,y)\big) dxdy 
	\end{align}
	
	Using coordinate transformation, $(u,v)=L_{n}^{-1}(x,y),$ the above expression becomes,

	\begin{align}
		\nonumber 
		M&=\sum_{n=1}^{N} \int_{D_{n}} F_{n}\big(L_{n}^{-1}(x,y), foL_{n}^{-1}(x,y)\big) dxdy \\ \nonumber
		&=\sum_{n=1}^{N} \int_{D} F_{n}\big(u,v,f(u,v)\big) \delta_{n}  dudv \\
		\nonumber
		&=\sum_{n=1}^{N}\int_{D} (\alpha_{n5}u+\alpha_{n6}v+\alpha_{n7}f(u,v)+\beta_{n3}) \delta_{n}dudv \\
		\nonumber
		& =\sum_{n=1}^{N}\int_{D} (\alpha_{n5}u+\alpha_{n6}v+\beta_{n3}) \delta_{n}dudv + \sum_{n=1}^{N}\int_{D} \alpha_{n7} f(u,v) \delta_{n} dudv \\
		\nonumber
		&=B+\sum_{n=1}^{N} \alpha_{n7} \delta_{n} \int_{D} f(u,v) dudv \\
		\nonumber
		&=B+AM	
	\end{align}
	
	\noindent	where $B=\sum_{n=1}^{N}\int_{D} (\alpha_{n5}u+\alpha_{n6}v+\beta_{n3}) \delta_{n}dudv,$  $A=\sum_{n=1}^{N} \alpha_{n7} \delta_{n}$  and $\delta_{n}=det(L_{n}),$
	which implies that 
	
	\begin{align}
		\nonumber
		M&=B+AM \\
		\nonumber
		M(1-A)&=B \\	
		M &= \frac{B}{1-A}
	\end{align}

	\section{Relation of the BFIF with the equation of the plane through the vertices of $D_{n}$}
	\begin{theorem}
		Let $f$ be the bivariate fractal interpolation function to the given data set. Then, $f$ satisfies the relation  
		\begin{align}
			f(x,y)&=h(x,y)+\alpha_{n7}(f-b)oL_{n}^{-1}(x,y)
		\end{align} where $h$ is the equation of the plane passing through the vertices of $D_{n}$ and $b$ is the equation of the plane passing through the vertices of $D.$
	\end{theorem}
	\begin{proof}
		Consider $h,b$ as the equations of the planes passing through the vertices of $D_{n}$ and $D$ respectively. $i.e,$ 
		\begin{small}
			\begin{align}
				h(x,y) &= \frac{ \big(x_{n2}z_{n3}-x_{n3}z_{n2}+x_{n3}z_{n1}-x_{n1}z_{n3}+x_{n1}z_{n2}-x_{n2}z_{n1}\big)y}{\big(x_{n2}y_{n3}-x_{n3}y_{n2}+x_{n3}y_{n1}-x_{n1}y_{n3}+x_{n1}y_{n2}-x_{n2}y_{n1}\big)} \\ \nonumber &- \frac{\big(y_{n2}z_{n3}-y_{n3}z_{n2}+y_{n3}z_{n1}-y_{n1}z_{n3}+y_{n1}z_{n2}-y_{n2}z_{n1}\big)x}{\big(x_{n2}y_{n3}-x_{n3}y_{n2}+x_{n3}y_{n1}-x_{n1}y_{n3}+x_{n1}y_{n2}-x_{n2}y_{n1}\big)} \\ \nonumber & - \frac{\big[\big(x_{n3}y_{n2}-x_{n2}y_{n3}\big)z_{n1}+ \big(x_{n1}y_{n3}-x_{n3}y_{n1}\big)z_{n2}+\big(x_{n2}y_{n1}-x_{n1}y_{n2}\big)z_{n3}\big] }{\big(x_{n2}y_{n3}-x_{n3}y_{n2}+x_{n3}y_{n1}-x_{n1}y_{n3}+x_{n1}y_{n2}-x_{n2}y_{n1}\big)}
			\end{align} 
		\end{small}
		and 
		\begin{align}
			b(x,y) &= \frac{ \big(x_{2}z_{3}-x_{3}z_{2}+x_{3}z_{1}-x_{1}z_{3}+x_{1}z_{2}-x_{2}z_{1}\big)y}{\big(x_{2}y_{3}-x_{3}y_{2}+x_{3}y_{1}-x_{1}y_{3}+x_{1}y_{2}-x_{2}y_{1}\big)} \\ \nonumber & - \frac{\big(y_{2}z_{3}-y_{3}z_{2}+y_{3}z_{1}-y_{1}z_{3}+y_{1}z_{2}-y_{2}z_{1}\big)x}{\big(x_{2}y_{3}-x_{3}y_{2}+x_{3}y_{1}-x_{1}y_{3}+x_{1}y_{2}-x_{2}y_{1}\big)} \\ \nonumber
			& - \frac{\big[\big(x_{3}y_{2}-x_{2}y_{3}\big)z_{1}+ \big(x_{1}y_{3}-x_{3}y_{1}\big)z_{2}+\big(x_{2}y_{1}-x_{1}y_{2}\big)z_{3}\big] }{\big(x_{2}y_{3}-x_{3}y_{2}+x_{3}y_{1}-x_{1}y_{3}+x_{1}y_{2}-x_{2}y_{1}\big)}.
		\end{align} Then, the function $Q_{n}$ satisfies 
		
		\begin{align*}
			Q_{n}(x,y)=hoL_{n}(x,y)-\alpha_{n7}b(x,y),
		\end{align*} which implies 
		\begin{align}
			Q_{n}oL_{n}^{-1}(x,y)=h(x,y)-\alpha_{n7}boL_{n}^{-1}(x,y). \label{eqn:apr}
		\end{align}
		Substituting \eqref{eqn:apr} in \eqref{rec}, it is obtained that 
		\begin{align*}
			f(x,y)=h(x,y)+\alpha_{n7}(f-b)oL_{n}^{-1}(x,y).
		\end{align*} Hence the proof.
	\end{proof}
	\section{Selection of scaling factors}
	\noindent 	Let $R=\{(x_{nj},y_{nj},z_{nj}):n=1,2,...,N, j=1,2,3\}$ be the given set of data. Consider the points $(x_{nl}, y_{nl})$ in the subtriangle $D_{n}$ where $l=1,2,...,S, S=
	(d^{'}+1)^{2}, d^{'}$ is the number of subdivisions along the height of the triangle $D_{n}$ for $n=1,2,...,N.$\\
	
	If $g$ is the fractal interpolation function to this data set, then 
	\begin{align*}
		z_{nl} &=g(x_{nl}, y_{nl}) \\
		&=\alpha_{n7}goL_{n}^{-1}(x_{nl},y_{nl}) + Q_{n}oL_{n}^{-1}(x_{nl},y_{nl}) \\
		&=\alpha_{n7}goL_{n}^{-1}(x_{nl},y_{nl}) + h(x_{nl},y_{nl})-\alpha_{n7}boL_{n}^{-1}(x_{nl},y_{nl})
	\end{align*}
	Now, approximating $g$ by $h$ and considering the optimization problem 
	\begin{align*}
		\text{min} \,\, E(\alpha_{n7})=\sum_{l=1}^{S} \Big[z_{nl}-h(x_{nl},y_{nl})-\alpha_{n7}(h-b)oL_{n}^{-1}(x_{nl},y_{nl})\Big]^2
	\end{align*} with $-0.9 \leq \alpha_{n7} \leq 0.9,$
	then, the minimum value of $\alpha_{n7}$ is :	
	\begin{align}
		\alpha_{n7}=\frac{\sum_{l=1}^{S} \Big[z_{nl}-h(x_{nl},y_{nl})\Big] u_{nl}}{\sum_{l=1}^{S}(u_{nl})^2},
	\end{align}
	where $u_{nl}=(h-b)oL_{n}^{-1}(x_{nl},y_{nl}).$
	\vspace{.4cm}
	\\
	The upper bound for $\alpha_{n7}$ can be calculated by applying Cauchy-Schwartz inequality to the above value of $\alpha_{n7},$ 
	\begin{align*}
		|\alpha_{n7}| &\leq \frac{\Big[ \sum_{l=1}^{S} \big[|z_{n}-h(x_{nl},y_{nl})|\big]^{2}\Big]^{1/2} \Big[ \sum_{l=1}^{S} |u_{nl}|^2\Big]^{1/2}} {\sum_{l=1}^{S} |u_{nl}|^2}
	\end{align*}
	Put $k_{h}=\Big[ \sum_{l=1}^{S} |u_{nl}|^2\Big]^{1/2}.$ Then, 
	\begin{align}
		|\alpha_{n7}| \leq \frac{\Big[ \sum_{l=1}^{S} |z_{nl}-h(x_{nl},y_{nl})|^2\Big]^{1/2}}{k_{h}}
	\end{align}
	
	\noindent	However, the formula given below has been used for the computation purpose. 
	\begin{align}
		\alpha_{n7}=\frac{Z_{G}-\frac{1}{3}(z_{n1}+z_{n2}+z_{n3})}{Z_{H}-\frac{1}{3}(z_{1}+z_{2}+z_{3})},
	\end{align}
	\noindent  where $Z_{G}=$ value of $f$ at the centroid of the subtriangle $D_{n},$ $Z_{H}=$ value of $f$ at the centroid of $D.$

	\section{Error analysis}
	Let $f$ be a continuous function on $D.$ Then, 
	\begin{align*}
		||f||_{\infty}=\text{max}\{|f(x,y)|: (x,y)\in D\}
	\end{align*}
	
	Modulus of continuity of $f$ is defined as 
	\begin{align*}
		w_{f}(\delta)=\text{sup}\big\{|f(x,y)-f(x^{'},y^{'})|:(x,y), (x^{'},y^{'})\in D, d\big((x,y), (x^{'},y^{'})\big)\leq \delta\big\}
	\end{align*}
	
	\begin{lemma}
		If $f$ is a continuous function providing the data $R=\{(x_{nj},y_{nj},z_{nj}): n=1,2,...,N, j=1,2,3\}$ and $g$ be the corresponding fractal interpolation function with scale vector $\overline{\alpha_{n7}}.$ Then, 
		\begin{align*}
			||f-g||_{\infty} \leq w_{f}(\delta)k{'}+\frac{||h-b||_{\infty}|\alpha_{n7}|_{\infty}}{1-|\alpha_{n7}|_{\infty}}.
		\end{align*}
	\end{lemma}
	\begin{proof}
		Let $h$ be the equation of a plane passing through $(x_{n1},y_{n1}, z_{n1}), (x_{n2},y_{n2}, z_{n2}), (x_{n3},y_{n3}, z_{n3}).$ Then, \begin{align*}
			||f-g||_{\infty} \leq ||f-h||_{\infty} + ||h-g||_{\infty}
		\end{align*} Consider the first part. \vspace{0.1cm} \\		
		Let $w_{f}(\delta)$ be the modulus of continuity of $f,$ $i.e,$ 
		\begin{align*}
			w_{f}(\delta)=\text{sup}\big\{|f(x,y)-f(x^{'},y^{'})|:(x,y), (x^{'},y^{'})\in D, d\big((x,y), (x^{'},y^{'})\big)\leq \delta\big\}
		\end{align*}
		Now, rearranging  $h(x,y),$ 
		\begin{align*}
			h(x,y)&=\frac{z_{n1}\big(x_{n3}y-xy_{n3}+x_{n2}y_{n3}-x_{n3}y_{n2}+xy_{n2}-x_{n2}y\big)}{\big(x_{n2}y_{n3}-x_{n3}y_{n2}+x_{n3}y_{n1}-x_{n1}y_{n3}+x_{n1}y_{n2}-x_{n2}y_{n1}\big)} \\ & + \frac{z_{n2}\big(xy_{n3}-x_{n3}y+x_{n3}y_{n1}-x_{n1}y_{n3}+x_{n1}y-xy_{n1}\big)}{\big(x_{n2}y_{n3}-x_{n3}y_{n2}+x_{n3}y_{n1}-x_{n1}y_{n3}+x_{n1}y_{n2}-x_{n2}y_{n1}\big)}  \\ &+ \frac{z_{n3}\big(x_{n2}y-xy_{n2}+x_{n1}y_{n2}-x_{n2}y_{n1}+xy_{n1}-x_{n1}y\big) }{\big(x_{n2}y_{n3}-x_{n3}y_{n2}+x_{n3}y_{n1}-x_{n1}y_{n3}+x_{n1}y_{n2}-x_{n2}y_{n1}\big)} 
		\end{align*}
		Now, $|f(x,y)-h(x,y)|$
		\begin{align} 
			\nonumber
			&\leq 2\text{sup}\big\{ |f(x,y)-f(x_{n1},y_{n1})|\big\}  \text{max}_{n,j}\{|x_{nj}|\}max_{n,j}\{|y_{nj}|\} \\ \nonumber &+2\text{sup}\big\{|f(x,y)-f(x_{n2},y_{n2})|\big\}  \text{max}_{n,j}\{|x_{nj}|\}\text{max}_{n,j}\{|y_{nj}|\} \\  \nonumber &+ 2\text{sup}\big\{ |f(x,y)-f(x_{n3},y_{n3})|\big\}  \text{max}_{n,j}\{|x_{nj}|\}\text{max}_{n,j}\{|y_{nj}|\} \\ &+2 \text{sup}\big\{ |f(x_{n1},y_{n1})-f(x_{n2},y_{n2})|\big\}  \text{max}_{n,j}\{|x_{nj}|\}\text{max}_{n,j}\{|y_{nj}|\} \label{eqn:qwe}\\  \nonumber &+ 2\text{sup}\big\{ |f(x_{n1},y_{n1})-f(x_{n3},y_{n3})|\big\}  \text{max}_{n,j}\{|x_{nj}|\}\text{max}_{n,j}\{|y_{nj}|\} \\  \nonumber &+2 \text{sup}\big\{ |f(x_{n2},y_{n2})-f(x_{n3},y_{n3})|\big\} \text{max}_{n,j}\{|x_{nj}|\}\text{max}_{n,j}\{|y_{nj}|\} 
		\end{align}
		Here, $x,y,x_{nj},y_{nj} $ for $n=1,2,...,N, j=1,2,3 $ lie in $D_{n}.$ \vspace{0.1cm} \\ Take $\delta$ to be the maximum of the distance between any two points in $D_{n}.$ Then, by definition, 	\begin{align*}
			w_{f}(\delta)=\text{sup}\big\{|f(x,y)-f(x_{nj},y_{nj})|:(x,y), (x_{nj},y_{nj})\in D_{n}, d\big((x,y), (x_{nj},y_{nj})\big)\leq \delta\big\}
		\end{align*} 
		Therefore, \eqref{eqn:qwe} becomes 
		\begin{align*}
			|f(x,y)-h(x,y)|\leq 12 w_{f}(\delta) \text{max}_{n,j}\{|x_{nj}|\}\text{max}_{n,j}\{|y_{nj}|\}
		\end{align*} Put $k^{'}=12 \text{max}_{n,j}\{|x_{nj}|\}\text{max}_{n,j}\{|y_{nj}|\} .$ Then,  $|f(x,y)-h(x,y)|\leq w_{f}(\delta) k^{'}.$
		Hence, 
		\begin{align*}
			||f-h||_{\infty} \leq  w_{f}(\delta) k^{'}.
		\end{align*}
		Now, considering the second part, $||h-g||_{\infty},$ 
		Using the equation, 
		\begin{align*}
			g(x,y)=h(x,y)+\alpha_{n7}(g-b)oL_{n}^{-1}(x,y)
		\end{align*}
		
		Now,
		\begin{align*}
			||g-h||_{\infty}  &\leq |\alpha_{n7}|_{\infty} ||g-b||_{\infty} \\ & \leq \Big[ ||g-h||_{\infty} + ||h-b||_{\infty} \Big] |\alpha_{n7}|_{\infty} \\ &\leq \frac{||h-b||_{\infty}|\alpha_{n7}|_{\infty}}{1-|\alpha_{n7}|_{\infty}}
		\end{align*}   
		Therefore, \begin{align*}
			||f-g||_{\infty} \leq w_{f}(\delta)k{'}+\frac{||h-b||_{\infty}|\alpha_{n7}|_{\infty}}{1-|\alpha_{n7}|_{\infty}}.
		\end{align*} Hence the proof. 
	\end{proof}  
	
	\begin{theorem}
		If $f$ is a continuous function providing the data $R=\{(x_{nj},y_{nj},z_{nj}): n=1,2,...,N, j=1,2,3\}$ and $g$ be the corresponding fractal interpolation function with scale vector $\overline{\alpha_{n7}}.$ Then, the double integral calculated by the proposed method converges to the actual integral value as $\delta \rightarrow 0.$
	\end{theorem}
	
	\begin{proof}
		\begin{align*}
			|E| = \Big|\int \int _{D} f - \int \int_{D}g\Big| \leq \Delta ||f-g||_{\infty}
		\end{align*}
		\noindent where $\Delta$ is the area of the triangle $D$ and 
		\begin{align*}
			||f-g||_{\infty} \leq w_{f}(\delta)k^{'}+\frac{||h-b||_{\infty}|\alpha_{n7}|_{\infty}}{1-|\alpha_{n7}|_{\infty}}.
		\end{align*}
		\noindent Since $f$ is an interpolation function to the data set, 
		\begin{align*}
			|z_{nl}-h(x_{nl},y_{nl})| &= |f(x_{nl},y_{nl})-h(x_{nl},y_{nl})| \\ &\leq \text{sup}_{l} \big\{|f(x_{nl},y_{nl})-h(x_{nl},y_{nl})|\big\} \\ &=||f-h||_{\infty} \\ &\leq w_{f}(\delta) k^{'}
		\end{align*}
		Now, 
		\begin{align*}
			|\alpha_{n7}|_{\infty} &\leq \frac{\Big[{\sum_{l=1}^{S} |z_{nl}-h(x_{nl},y_{nl})|^2 \Big]^{1/2}}}{k_{h} } \\ &\leq \frac{\Big[ \sum_{l=1}^{S} (w_{f}(\delta)k^{'})^2 \Big]^{1/2}}{k_{h}} \\  &=\frac{w_{f}(\delta)k^{'}S}{k_{h}} 
		\end{align*} Writing $k_{h}^{'}=\frac{Sk^{'}}{k_{h}},$ implies $|\alpha_{n7}|_{\infty} \leq w_{f}(\delta)k_{h}^{'}.$

		Since $\alpha_{n7}$ is bounded by $w_{f}(\delta)k_{h}^{'},$

		\begin{align*}
			||f-g||_{\infty} \leq  w_{f}(\delta)k^{'}+\frac{||h-b||_{\infty}w_{f}(\delta)k_{h}^{'}}{1-w_{f}(\delta)k_{h}^{'}}.
		\end{align*} For a continuous function $f$ on $D,$ as $\delta \rightarrow 0, w_{f}(\delta)\rightarrow 0.$  \\
		$i.e, ||f-h|| \leq w_{f}(\delta) k^{'} \rightarrow 0, $ implies $ h$ uniformly converges to $f$ and 
		\begin{align*}
			|E| &\leq \Delta ||f-g||_{\infty} \\ & = \Delta w_{f}(\delta) \Big[ k^{'} + \frac{k_{h}^{'}||h-b||_{\infty}}{1-w_{f}(\delta)k_{h}^{'}}\Big] \rightarrow 0
		\end{align*} as $w_{f}(\delta) \rightarrow 0.$ Hence the proof.
	\end{proof}
	\section{Examples}
	
	\noindent	The double integral values and the attractors of the IFS are provided for two functions. The computation of the attractor and the integral value is done using amrita-hpc matlab 2019. \vspace{0.3cm}\\
	\noindent \textbf{Example 1 : Matyas Function} \vspace{0.1cm}\\	
	\noindent Consider Matyas function  \begin{align*}
		f(x,y)=0.26(x^{2}+y^{2})-0.48xy \,\,\, \text{where} -10 \leq x \leq 10, -10 \leq y \leq 10.
	\end{align*} The actual integral value of the double integral ($I$) is compared with the numerical integration method proposed. The comparison is given in Table \ref{Matyas function}. 
	The attractor of the IFS is shown in Figure \ref{Figure 2}.
	\begin{center}
		\begin{figure}[h]
			\begin{center}
				\includegraphics[width=8cm]{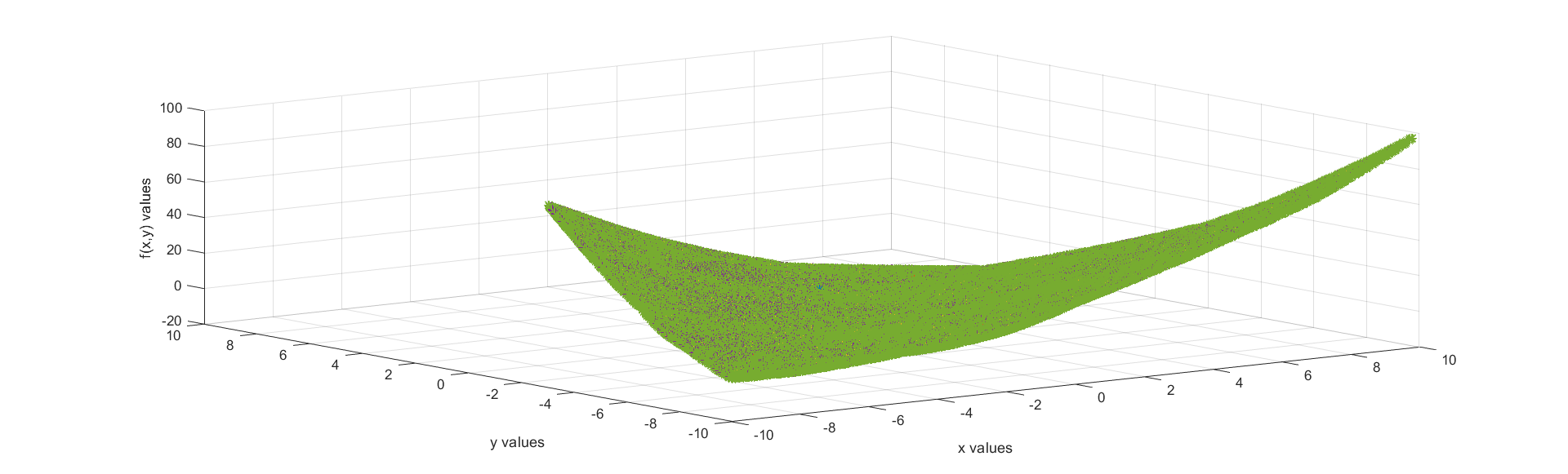}
				\caption{Attractor of the IFS for Matyas function}
				\label{Figure 2}
			\end{center}
		\end{figure}
	\end{center}
	\FloatBarrier
	\begin{table}[htb]
		\begin{tabular}{|c|c|c|c|c|}
			\hline
			d(No of subdivisions) & N (No of subtriangles) & M & I & Error (M-I) \\
			\hline
			4 & 27 & 2.4299e+03 & 2.6000e+03 & -170.0594 \\
			\hline
			7 &  87  &    2.5401e+03       & 2.6000e+03 &    -59.8738       \\
			\hline
			10 & 183 & 2.5696e+03 & 2.6000e+03 & -30.3787 \\
			\hline
			13 & 315     &    2.5818e+03        & 2.6000e+03 &   -18.2392      \\
			\hline
			\vdots & \vdots & \vdots & \vdots &\vdots \\
			
			\hline
			73 &    10515     &  2.5993e+03     & 2.6000e+03 &  -0.6064    \\
			\hline
			76 & 11403 & 2.5994e+03 & 2.6000e+03 & -0.5598\\		
			\hline
			79 &   12327     &  2.5994e+03          & 2.6000e+03 &    -0.5182     \\
			\hline
			\vdots & \vdots & \vdots & \vdots &\vdots \\
			\hline
			148 &   43515   &    2.5999e+03       &  2.6000e+03 &    -0.1292        \\
			\hline
			151 &    45303    &   2.5999e+03      & 2.6000e+03 &    -0.0786        \\
			\hline
			154 &   47127     &  2.5999e+03        & 2.6000e+03 &    -0.0523        \\
			\hline
			
		\end{tabular}
		\captionsetup{justification=centering}
		\caption{Matyas function\label{Matyas function}}
	\end{table}
	\FloatBarrier

	\noindent \textbf{Example 2 : Three-hump Camel Function} \vspace{0.1cm}\\
	\noindent Consider Three-hump Camel function $$f(x,y)=2x^{2}-1.05x^{4} + \frac{x^{6}}{6} +xy+ y^{2}\,\,\, \text{where} -5 \leq x \leq 5, -5 \leq y \leq 5.$$ A similar comparison of the integral value as in example one is given in Table \ref{Threehump camel function}.
	The attractor of the IFS is given in Figure \ref{Figure 3}.
	
	\begin{center}
		\begin{figure}[h]
			\begin{center}
				\includegraphics[width=10cm]{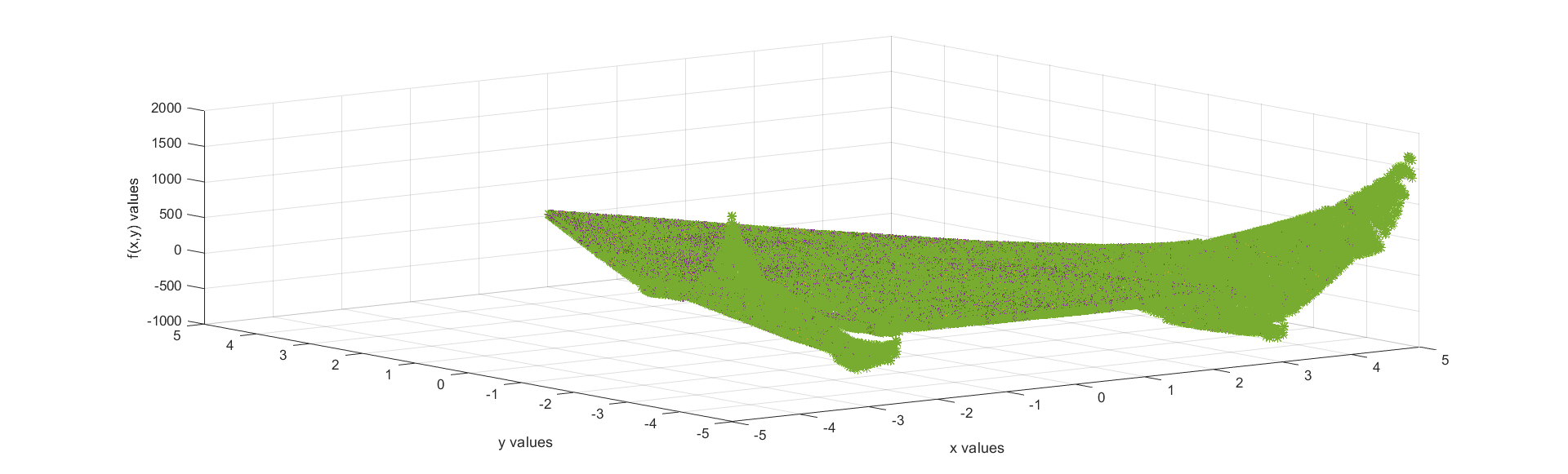}
				\caption{Attractor of the IFS for Three-hump camel function}
				\label{Figure 3}
			\end{center}
		\end{figure}
		
	\end{center}

	\begin{table}[h]
		\begin{center}
			\begin{tabular}{|c|c|c|c|c|}
				\hline 		
				d(No of subdivisions) & N (No of subtriangles) & M & I & Error (M-I) \\
				\hline
				4 & 27 & 2.1979e+03 & 3.2961e+03 &  -1.0982e+03\\
				\hline
				7 & 87 &   2.8932e+03          & 3.2961e+03 &     -402.9155         \\
				\hline
				10 & 183 & 3.0917e+03& 3.2961e+03 &-204.4563 \\
				\hline
				13 & 315 &    3.1732e+03       & 3.2961e+03 &  -122.9447          \\
				\hline
				\vdots & \vdots & \vdots & \vdots & \vdots \\
				\hline
				73 &   10515 &  3.2920e+03     & 3.2961e+03 &    -4.0326       \\
				\hline
				
				76 & 11403   & 3.2924e+03  & 3.2961e+03 & -3.7214  \\
				\hline
				79 & 12327 &  3.2926e+03          &  3.2961e+03 &  -3.4448    \\
				\hline
				\vdots & \vdots & \vdots & \vdots & \vdots \\
				\hline
				148 & 43515 &      3.2952e+03       & 3.2961e+03 &    -0.8926         \\
				\hline
				151 &   45303     &    3.2960e+03        &   3.2961e+03           &   -0.0624         \\
				\hline
				154 & 47127 &    3.2961e+03        & 3.2961e+03 &   -0.0241        \\
				\hline
			\end{tabular}
			\caption{Three-hump camel function \label{Threehump camel function}}
		\end{center}
	\end{table}

	\FloatBarrier
	
	\section{Conclusion}
	This paper describes the construction of bivariate fractal interpolation functions using the coloring technique and derives the formula for double integration. In between, a novel method is also proposed for the partition of the triangle.
	Since the IFS itself induces a well defined fractal operator, the paper establishes that the graph of the function coincides with the attractor of the IFS. Instead of choosing the vertical scaling factor randomly, this paper provides a useful formula for it, based on the shape of the interpolating domain. Further, this paper shows that the approximating function for the bivariate fractal interpolation functions over triangular regions is nothing but the equation of the plane passing through the vertices of a triangle. This function is then used to prove the theorems in error analysis. Finally, the results of the double integration are tabulated and the method of construction is explained with appropriate graphs.

\end{document}